\documentclass[a4paper,11pt]{article}
\usepackage{amsmath}
\usepackage{amssymb}
\usepackage{mathrsfs}
\usepackage{amsfonts}
\usepackage{amsthm}

\makeatletter 
\@addtoreset{equation}{section}
\makeatother  

\newcommand\abs[1]{\lvert#1\rvert}

\newcommand\RR{\ensuremath{\mathbb{R}}}

\title{Floquet bundles for tridiagonal competitive-cooperative systems with Applications}
\begin{document}
\author {
Chun Fang\\
Department of
Mathematics and Statistics\\
University of Helsinki, FIN-00014, Finland\\
E-mail: chun.fang@helsinki.fi
\\
\\
Mats Gyllenberg\\
Department of
Mathematics and Statistics\\
University of Helsinki, FIN-00014, Finland\\
E-mail: mats.gyllenberg@helsinki.fi
\\
\\
Yi Wang\thanks{Partially supported by NSF of China No.10971208, and
  the Finnish Center of Excellence
in Analysis and Dynamics.}\\
$^a$Department of Mathematics\\
 University of Science and Technology of China
\\ Hefei, Anhui, 230026, P. R. China
\\
$^b$Department of
Mathematics and Statistics\\
University of Helsinki, FIN-00014, Finland
\\
E-mail: wangyi@ustc.edu.cn\\
}
\date{}

\maketitle

\begin{abstract}
   For a general time-dependent linear competitive-cooperative tridiagonal  system of differential equations, we obtain canonical Floquet invariant bundles which are exponentially separated in the framework of skew-product flows. Such Floquet bundles naturally reduce to the standard Floquet space when the system is assumed to be time-periodic. The obtained Floquet theory is applied to study the dynamics on the hyperbolic omega-limit sets for the nonlinear competitive-cooperative tridiagonal  systems in time-recurrent structures including almost periodicity and almost automorphy.
\end{abstract}

\section{Introduction}
In this paper we study the dynamical properties of systems of differential equations with a tridiagonal structure. To be precise, this refers to systems of the form
\begin{equation}\label{tri-equation-f1}
\begin{split}
\dot{x}_1 &=f_1(t,x_1,x_2),\\
\dot{x}_i &=f_i(t,x_{i-1},x_i,x_{i+1}),\quad 2\leq i\leq n-1,\\
\dot{x}_n &=f_n(t,x_{n-1},x_n).\\
\end{split}
\end{equation}
We assume the nonlinearity $f=(f_1,f_2,\cdots,f_n)$ is defined on $\RR\times \RR^n$ and $C^1$-{\it admissible}, i.e.
$f$ together with its first derivatives with respect to $x=(x_1,x_2,\cdots,x_n)$,  is bounded and
 uniformly continuous on $\RR\times K$ for any compact set $K\subset \RR^n$. For the interpretation of \eqref{tri-equation-f1}, one can think of a hierarchy of
species $x_1,x_2\cdots,x_n$, where $x_i$ is the density or biomass of the ith species. In this
hierarchy, $x_1$ only interacts with $x_2$, $x_n$ only with $x_{n-1}$, and for $i = 2,\cdots,n-1$, $x_i$
interacts with $x_{i-1}$ and $x_{i+1}$. Such a hierarchy may occur in the water column of
an ocean or on a steep mountain side or on island groups, where each species dominates a species
zone (depth, altitude or different island) but is obliged to interact
with other species in the (narrow) overlap of their zones of dominance.

 Our key assumption about the tridiagonal system \eqref{tri-equation-f1} is that the variable $x_{i+1}$ forces
 $\dot{x}_i$, as well as $x_{i}$ forces
 $\dot{x}_{i+1}$, monotonically in the same fashion. That is, there are $\varepsilon_0>0$ and $\delta_i\in \{-1,+1\}$, such that
\begin{equation*}
\noindent {\bf (F)} \qquad \delta_i\dfrac{\partial f_i}{\partial x_{i+1}}(t,x)\ge \varepsilon_0,
\quad \delta_i\dfrac{\partial
f_{i+1}}{\partial x_i}(t,x)\ge \varepsilon_0, \qquad 1\le i\le n-1,
\end{equation*}
for all $(t,x)\in \RR\times \RR^n$. If $\delta_i=-1$ for
all $i$, then (\ref{tri-equation-f1}) is called {\it competitive}.
If $\delta_i=1$ for all $i$, then (\ref{tri-equation-f1}) is called
{\it cooperative}.  Here we do not
allow predator-prey relationship (see, e.g. \cite{AGT}).

We introduce new variables, following Smith
\cite{Smith}. Let $\hat{x}_i=\mu_ix_i,\mu_i\in\{+1,-1\},1\le
i\le n,$ with $\mu_1=1,\mu_i=\delta_{i-1}\mu_{i-1}$. Then the
system (\ref{tri-equation-f1}) transforms into a new system of the
same type with new
\[
\hat{\delta}_i=\mu_i\mu_{i+1}\delta_i=\mu_i^2\delta_i^2=1.
\]
Therefore we can always assume, without loss of generality, that the trdiagonal
system (\ref{tri-equation-f1}) is in fact cooperative, which means that
\begin{equation}\label{assumption}
\noindent \dfrac{\partial f_i}{\partial x_{i+1}}(t,x)\ge \varepsilon_0,
\ \dfrac{\partial
f_{i+1}}{\partial x_i}(t,x)\ge \varepsilon_0, \qquad 1\le i\le n-1, \,(t,x)\in \RR\times \RR^n.
\end{equation}
\noindent  In particular, if the system \eqref{tri-equation-f1}+\eqref{assumption} is linear, we write the corresponding system as the following form:  \begin{equation}\label{tri-diagonal-lin}
\begin{split}
  \dot{x}_1 &= a_{11}(t)x_1+a_{12}(t)x_2,\\
  \dot{x}_i &= a_{i,i-1}(t)x_{i-1}+a_{ii}(t)x_i+a_{i,i+1}(t)x_{i+1},\qquad 2\leq i\leq n-1, \\
  \dot{x}_n &= a_{n,n-1}(t)x_{n-1}+a_{nn}(t)x_n,
\end{split}
\end{equation}
where $a_{i,i+1}(t), a_{i+1,i}(t)\geq\varepsilon_0$, for all $t\in \RR$ and $1\leq i\leq n-1$.

When the linear system \eqref{tri-diagonal-lin} is time-periodic in $t$, Smith \cite{Smith} studied the Floquet theory by using an integer-valued Lyapunov function $\sigma$, first defined by Smillie \cite{Smillie} (see also similar forms by Mallet-Paret and Smith \cite{MS} and Mallet-Paret and Sell \cite{MSe}), and related the values of $\sigma$ to the Floquet multipliers for such linear perodic system. This function $\sigma$ is only continuously defined on an open and dense subset $\Lambda$ of $\RR^n$ (See section \ref{Floquet_solutions}). However, $\sigma$
is well defined for all but an at most finite set of points $t$ along a nontrivial solution of the
linear system \eqref{tri-diagonal-lin}, is locally constant near points where it is defined, and strictly decreasing
as $t$ increases through points where it is not defined. As a consequence, in a certain sense, $\sigma$ can be seen more or less as the discrete analog (although $\sigma$ still possesses somewhat difference) of zero-crossing number of Matano \cite{Ma} (discovered originally by Nickel \cite{Ni}) for scalar reaction-diffusion equations. By utilizing the zero-crossing number, Chow, Lu and Mallet-Paret \cite{CLM2} have already established the Floquet theory for linear periodic scalar parabolic equations.

In the first part of the present paper, we will develop a Floquet theory for a general linear time-dependent system of \eqref{tri-diagonal-lin}, and we express this theory in the language of invariant vector bundles and the so-called exponential separation (see, e.g. \cite{JMSW,Palmer,Po} and references therein). Our approach is fully motivated by the work of Chow, Lu and Mallet-Paret \cite{CLM2,CLM} for time-dependent scalar parabolic equations, and extends earlier work on linear autonomous equations in \cite{Smillie}  and linear time-periodic  equations in \cite{Smith}.

Concretely, for each $0\le m\le n-1$, we associate with a nontrivial solution $x_m(t)$ of \eqref{tri-diagonal-lin} (unique up to constant multiple) with $x_m(t)\in \Lambda$ and $\sigma(x_m(t))=m$ for all $t\in \RR$. These solutions are then treated as a base to decouple \eqref{tri-diagonal-lin} into a system of one-dimensional ordinary differential equations. Moreover, if one writes $W_m$ as the one-dimensional span of $x_m(0)$, we can show that $W_m$ varies continuously with respect to the coefficients of \eqref{tri-diagonal-lin} in certain appropriate topology, and hence, $W_m(\cdot)$ forms a one-dimensional vector bundle (called {\it Floquet bundle of} \eqref{tri-diagonal-lin}) over a proper product space, which will be indicated explicitly later. Moreover, the exponentially separated property is admitted between the various time-dependent Floquet bundles, by which one can obtain a more delicate decomposition of invariant bundles than those induced by Sacker-Sell \cite{SS,SS2} for linear skew-product flows.

 The Floquet bundles obtained here are analogous to ones obtained in \cite{CLM} for time-dependent scalar parabolic equations. However, as the function $\sigma$ is only defined (continuously) on $\Lambda$, not on the whole $\RR^n\setminus\{0\}$ (while the zero-crossing number can be defined on the whole phase space $X$ except for $\{0\}$), it needs more delicate and technical trick to construct the Floquet bundles. Moreover, among other things, one may also run into difficulties when dealing with the phase-points at which the integer-valued Lyapunov function strictly decreases. For zero-crossing number, it entails that such a phase-point, say $u\in X$,  possesses a multiple zero, and hence, one can employ the standard characteristic of $u(\xi)=u_{x}(\xi)=0$, for some $\xi$, to analyze such critical situation (see the detail in \cite[Corollary 4.8 and Theorem 5.1]{CLM}). In our case, the critical point occurs at $x\notin \Lambda$. Nevertheless, there is no similar useful characteristic of such a critical point; and for example, this is indeed the situation we encounter (see, e.g. Proposition \ref{linear_space} below), in which we are trying to show the linear property of the Floquet space. We overcome such difficulty by constructing a pair of so called relevant-sequences simultaneously and accomplish our approaches by analyzing the convergent property of one of such relevant-sequences.

It is well known that the linear theory of invariant bundles plays a crucial role in the study of qualitative properties of nonlinear differential equations. In the second part of this paper, we will investigate the nonlinear tridiagonal system \eqref{tri-equation-f1}+ \eqref{assumption} via the Floquet theory developed in the first part of our paper.

 To be more specific, we embed (\ref{tri-equation-f1}) into
the skew-product flow $\Pi_t:\RR^n\times H(f)\to \RR^n\times H(f),$
\begin{equation*}
\Pi_t(x_0,g)\mapsto (\pi(t,x_0,g),g\cdot t),
\end{equation*}
 where $\pi(t,x_0,g)$
is the solution of
\begin{equation*}
\begin{split}
\dot{x}_1 &=g_1(t,x_1,x_2),\\
\dot{x}_i &=g_i(t,x_{i-1},x_i,x_{i+1}),\quad 2\leq i\leq n-1;\\
\dot{x}_n &=g_n(t,x_{n-1},x_n),\\
\end{split}
\end{equation*}
with $\pi(0,x_0,g)=x_0\in \RR^n$, and $g=(g_1,\cdots,g_n)\in H(f)$,
$(g\cdot t)(\cdot,\cdot)=g(t+\cdot,\cdot)$. Here
$$H(f)\triangleq{\rm cl}\{f\cdot{\tau}|\tau\in \RR, f\cdot{\tau}(t,u)=f(t+\tau,u)\}$$ is called the hull of $f$,
where the closure is taken in the compact open topology (see \cite{Sell}). Obviously, $\pi$ satisfies
the cocycle property, i.e, $\pi(t+s,x,g)=\pi(s,\pi(t,x,g),g\cdot
t)$, for all $s,t\in \RR$ and $g\in H(f)$.

Since $f$ is $C^1$-admissible, the time-translation flow $g\cdot t$ on $H(f)$ is compact. Here we further assume that
the flow on $H(f)$ is even {\it
recurrent} or {\it minimal}. This is satisfied, for instance, when
$f$ is a {\it uniformly almost periodic}, or, more generally, {\it
a uniformly almost automorphic function} (see Definitions \ref{a-p-a-a-function}).

In \cite{Wang}, one of the present authors has shown that any minimal invariant set of $\Pi_t$ is an almost $1$-cover of $H(f)$ (see definitions in
Section \ref{application}),
and every $\omega$-limit set $\omega(x,g)$ of $\Pi_t$
contains at most two minimal sets. Moreover, it was also shown in \cite{Wang} that if the  $\omega$-limit set is distal or uniformly stable, then it is a $1$-cover of $H(f)$.

Inspired by the series work of Shen and Yi \cite{SY,SY2}, we will utilize the obtained Floquet theory to improve the lifting property of the $\omega$-limit sets which are hyperbolic (see Definition \ref{hyper-def} and Theorem \ref{hy-omega}), i.e. any hyperbolic $\omega$-limit set is also a $1$-cover of $H(f)$.

Our results here are natural generalization of the results of
Smillie \cite{Smillie} and Smith \cite{Smith}. Moreover, in a certain
sense, our results also extend to higher dimension ($n\ge 3$) the results in spatially
homogeneous cases by Hetzer and Shen \cite{HS}, who
investigated the dynamics of two-dimensional competitive or
cooperative almost periodic systems. In particular,  for $2$-D competitive or
cooperative almost periodic systems, we have obtained that any hyperbolic $\omega$-limit set is exactly a $1$-cover of the base $H(f)$ (see Corollary \ref{hy-omega-1}). This generalizes the results of
de Mottoni and Schiaffino \cite{MS} and Hale and Somolinos \cite{JHS}, who proved that all
solutions of two-dimensional $T$-periodic competitive or cooperative systems are
asymptotic to $T$-periodic solutions. See also \cite{HS2} and \cite{SW} for extensions of this work.

The paper is organized as follows. The Floquet solutions and spaces of system \eqref{tri-diagonal-lin} are constructed in Section 2 by taking certain limits of periodic linear tridiagonal systems. Moreover, we also relate the values of $\sigma$ to the Floquet solutions, and decouple \eqref{tri-diagonal-lin} into a system of one-dimensional ODEs. In section 3, we define the Floquet bundles and prove the resulting exponential separations between these invariant bundles in the language of the skew-product flow. Finally, we focus on nonautonomous nonlinear cooperative-competitive tridiagonal ODEs in Section 4; and study the lifting properties of hyperbolic omega-limit sets by using the Floquet Theory obtained in the previous sections.


\section{Floquet Solutions and Spaces}\label{Floquet_solutions}

In this section, we focus on the linear tridiagonal system \eqref{tri-diagonal-lin}, with all the coefficient functions being bounded and uniformly continuous on $\mathbb{R}$. Moreover, we assume that $a_{i,i+1}(t), a_{i+1,i}(t)\geq\varepsilon_0$, for all $t\in \RR$ and $1\leq i\leq n-1$.  Hereafter, we write the corresponding
 coefficient matrix $A(t)=(a_{ij}(t))_{n\times n}$. It is easy to see that matrix-valued function $A(t)$ is  tridiagonal and cooperative.

We will build up Floquet solutions and spaces for the general time-dependent linear system (\ref{tri-diagonal-lin}). Following \cite{Smillie,Smith}, we define a continuous map
\[\sigma:\ \Lambda\rightarrow\{0,1,2,\ldots{},n-1\}\]
on
$
\Lambda=\{x\in\mathbb{R}^n: x_1\neq 0, x_n\neq 0\text{ and if } x_i=0
\text{ for some}\ i,\ 2\leq i\leq n-1,\text{ then }x_{i-1}x_{i+1}<0\}
$
by
\begin{equation*}
\sigma(x)=\#\{i:x_i=0\text{ or }x_ix_{i+1}<0\}.
\end{equation*}
Here $\#$ denotes the cardinality of the set. Note that $\Lambda$ is open and dense in $\mathbb{R}^n$ and $\Lambda$ is the maximal domain on which $\sigma$ is continuous.

\newtheorem{dlf}{Lemma}[section]
\begin{dlf}\label{DLF}
Let $x(t)$ be a nontrivial solution of system \eqref{tri-diagonal-lin}. Then
\begin{itemize}
\item [\emph{(i)}] $x(t)\in\Lambda$ except possibly for isolated values of $t$;
\item [\emph{(ii)}] $\sigma(x(t))$ is nonincreasing as $t$ increases with $x(t)\in\Lambda$. Moreover, if $x(s)\notin\Lambda$ for some $s\in \mathbb{R}$ then $\sigma(x(s+))<\sigma(x(s-))$;
\item [\emph{(iii)}] $x(t)\in\Lambda$ and $\sigma(x(t))$ is a constant for $t$ sufficiently positive (resp. negative).

\end{itemize}
\end{dlf}

\begin{proof}
See \cite[Proposition 1.2]{Smith} for the proof of (i) and (ii). By virtue of (i) and (ii), $\sigma(x(t))$ can drop only finite times, which implies (iii).
\end{proof}

When $A(t)$ is periodic in $t$ with period $T$, the  Floquet multipliers of \eqref{tri-diagonal-lin} have the following property.

\newtheorem{periodic}[dlf]{Lemma}
\begin{periodic}\label{Smith}
\begin{itemize}
\item [\emph{(i)}] The $T$-periodic system \eqref{tri-diagonal-lin} has $n$ distinct positive Floquet multipliers $\alpha_{0}$,$\alpha_{1}$, $\ldots$,$\alpha_{n-1}$, satisfying
\begin{equation*}
\alpha_0>\alpha_1>\ldots>\alpha_{n-1}>0;
\end{equation*}
\item [\emph{(ii)}] If $E_{\alpha_m}$ are the corresponding one-dimensional eigenspaces associated with $\alpha_m$ then $E_{\alpha_m}\setminus \{0\}\subset\Lambda$ and
\begin{equation*}
\sigma(E_{\alpha_m}\setminus \{0\})=m,\qquad 0\leq m\leq n-1;
\end{equation*}
\item [\emph{(iii)}] Fix $0\leq m\leq n-1$, there exists a solution $x_m(t)$ such that
\begin{equation*}
\sigma(x_m(t))=m,\quad\text{for } t\in\mathbb{R}.
\end{equation*}
\end{itemize}
\end{periodic}

\begin{proof}
For (i) and (ii), see \cite[Theorem 1.3]{Smith}. We only  prove (iii). Fix $0\leq m\leq n-1$. Note that $\alpha_m$ is a positive number. By the standard Floquet theory, there exists a nontrivial solution of \eqref{tri-diagonal-lin} \begin{equation*}
x_{m}(t)=e^{\mu_mt}p_m(t),
\end{equation*}
where $\mu_m=\ln \alpha_m$ and $p_m(t)$ is a $T$-periodic function with $p_m(0)\in E_{\alpha_m}\backslash\{0\}$. Since $p_m(kT)=p_m(0)\in E_{\alpha_m}\backslash\{0\}$, (ii) implies that $p_m(kT)\in\Lambda$ and $\sigma(p_m(kT))=m$ for all $k\in\mathbb{Z}$. Combined by Lemma \ref{DLF}(iii), we readily get $x_m(t)\in\Lambda$ and $\sigma(x_m(t))=m$ for all $|t|$ sufficiently large. Then by Lemma \ref{DLF}(ii), we have $x_m(t)\in\Lambda$ and $\sigma(x_m(t))=m$ for all $t\in\mathbb{R}$, which completes the proof.
\end{proof}

The following proposition shows that Lemma \ref{DLF}(iii) can still hold for the general time-dependent system \eqref{tri-diagonal-lin}.

\newtheorem{solutionsbyzeros}[dlf]{Proposition}
\begin{solutionsbyzeros}\label{solutionsbyzeros}
For each $0\leq m\leq n-1$, there exists a solution $x_m(t)$
of \eqref{tri-diagonal-lin}
satisfying
\[\sigma(x_m(t))=m,\text{ for all } t\in\mathbb{R}.\]
\end{solutionsbyzeros}

\begin{proof}
We begin by constructing a sequence of continuous matrix-valued functions $\{A_k(t)\}_{k=1}^{\infty}$, where
\begin{equation*}
A_k(t)=\left\{
\begin{array}{ll}
(t+k+1)A(-k) & \text{if } -k-1<t<-k,\\
A(t) & \text{if } -k<t<k,\\
(k+1-t)A(k) & \text{if }\ k<t<k+1,
\end{array}
\right.
\end{equation*}
on $[-k-1,k+1]$; and moreover through a standard practice one can extend $A_k(t)$ to a $2(k+1)$-periodic function on $\mathbb{R}$. It is easily seen that $\{A_k(t)\}_{k=1}^{\infty}$ is uniformly bounded and $A_k(t)$ converges to $A(t)$ uniformly on any compact set in $\mathbb{R}$.

For each $k\geq 1$, consider the $2(k+1)$-periodic equation
\begin{equation*}
\dot{x}=A_k(t)x.
\end{equation*}
Note that $A_k(t)$ is of cooperative tridiagonal form. Then Lemma \ref{Smith}(iii) implies that for each $0\leq m\leq n-1$, there exists a solution $x_m^{(k)}(t)$ defined on $\mathbb{R}$ such that
\begin{equation*}
\sigma(x_m^{(k)}(t))=m,\quad\text{ for all }t\in\mathbb{R}.
\end{equation*}
We normalize these solutions so that the initial points satisfy $|x_m^{(k)}(0)|=1$.

Fix $0\leq m\leq n-1$ and consider the sequence $\{x_m^{(k)}(0)\}_{k=1}^{\infty}$. Then there exists a subsequence $\{k'\}$ such that $x_m^{(k')}(0)\rightarrow y_m$, with $|y_m|=1$, as $k'\rightarrow\infty$. Recall that $\{A_{k'}(t)\}$ is uniformly bounded and tends to $A(t)$ uniformly on any compact interval. By virtue of Grownwall inequality, the corresponding solution $x_m(t)$ of equation \eqref{tri-diagonal-lin} with initial values $x_m(0)=y_m$ satisfies that $x_m^{(k')}(t)\rightarrow x_m(t)$ uniformly on any compact interval as $k'\rightarrow\infty$.

We claim that $x_m(t)\in \Lambda$ and $\sigma(x_m(t))\equiv m$ for all $ t\in\mathbb{R}$. Indeed, by Lemma \ref{DLF}(iii), one can find a $t_0>0$, such that $x_m(t)\in\Lambda$ and $\sigma(x_m(t))=N_1$ (resp. $N_2$) for all $t\geq t_0$ (resp. $t\leq -t_0$). On the other hand, using the openness of $\Lambda$ and $\lim_{k\rightarrow\infty}x_m^{(k)}(t_0)=x_m(t_0)$, we obtain that $\sigma(x_m^{(k)}(t_0))=N_1$ and $\sigma(x_m^{(k)}(-t_0))=N_2$ for all $k$ sufficiently large. It then follows from Lemma \ref{Smith}(iii) that
\[N_1=N_2=m\quad\text{ for all }t\in\mathbb{R},\]
which implies that $x_m(t)\in\Lambda$ for all $t\in\mathbb{R}$. This completes the proof.\end{proof}

For integers $0\leq m\leq l\leq n-1$, we define the set
\begin{eqnarray*}
W_{m,l}(A)&=&\{x\in\mathbb{R}^n\setminus \{0\}|\text{the solution } x(t) \text{ of }
\eqref{tri-diagonal-lin} \text{ with } x(0)=x \nonumber \\
&& \text{ satisfies } m\leq\sigma(x(t))\leq l,  \text{ whenever } x(t)\in \Lambda \}\cup\{0\}.
\end{eqnarray*}
Clearly, Proposition \ref{solutionsbyzeros} implies that $W_{m,l}(A)$ is a nonempty set. Moreover, we have

\newtheorem{linear_space}[dlf]{Proposition}
\begin{linear_space}\label{linear_space}
The set $W_ {m,l}(A)$ is a linear subspace of $\mathbb{R}^n$; and \[\dim(W_{m,l}(A))=l-m+1.\]
\end{linear_space}

\begin{proof}
Take $x_0, y_0\in W_{m,l}(A)$. Let $x(t)$ and $y(t)$
 are nontrivial solutions of \eqref{tri-diagonal-lin} with $x(0)=x_0, y(0)=y_0$ such that $m\leq\sigma(x(t)),\sigma(y(t))\leq l$, whenever $x(t),y(t)\in \Lambda$. Since $\sigma(\alpha x)=\sigma(x)$ for any $x\in \Lambda$ and $\alpha\neq 0$, it suffices to show that \[m\leq\sigma(x(t)+y(t))\leq l,\quad\text{whenever }x(t)+y(t)\in \Lambda.\]
We only prove  the upper bound $l$, as the proof of the lower
bound is similar.

To this end, by (ii)-(iii) of Lemma \ref{DLF}, suppose that there exists a $t_0\in\mathbb{R}$ such that
\begin{equation}\label{x+y-lin}
x(t)+y(t)\in\Lambda\text{ and }\sigma(x(t)+y(t))>l, \quad\text{for all }t\le t_0.
\end{equation}
Since $\Lambda$ is an open set, there exist a $\delta_0\in (0,1)$ such that for any $\delta\in [0,\delta_0]$,
$x(t_0)+(1-\delta)y(t_0)\in\Lambda$ and $\sigma(x(t_0)+(1-\delta)y(t_0))=\sigma(x(t_0)+y(t_0))>l$
For each $\delta\in [0,\delta_0]$, since $x(t)+(1-\delta)y(t)$ is a solution of \eqref{tri-diagonal-lin}, Lemma \ref{DLF}(ii-iii) implies there exists a $t_{\delta}\leq t_0$ such that,
\begin{equation}\label{pertu-open}
\sigma(x(t)+(1-\delta)y(t))>l,\quad\text{for all } t\leq t_{\delta}.
\end{equation}

Now, let us fix a $\delta\in (0,\delta_0]$. By virtue of \eqref{x+y-lin}
and \eqref{pertu-open}, one can choose a sequence
$\{t_k\}_{k=1}^\infty,$ satisfying $\min\{t_0,t_{\delta}\}>t_k\to
-\infty$, such that
\begin{equation*}
\sigma(x(t_k)+(1-\delta)y(t_k))>l \text{ and } \sigma(x(t_k)+y(t_k))>l,
\end{equation*}
 for all $k\ge 1$.
Recall that $x_0,y_0\in W_{m,l}(A)$, one may further assume that
$\sigma(x(t_k)), \sigma(y(t_k))\leq l$ for all $k\ge 1$.
Consequently, for each $k\ge 1$, there exist $\delta_1(t_k)\in (0,1-\delta)$
and $\delta_2(t_k)\in (0,1)$ such that
\[x(t_k)+\delta_1(t_k)y(t_k) \notin \Lambda\quad\text{and }\quad\delta_2(t_k)x(t_k)+y(t_k) \notin \Lambda.\]
Now we define the following pair of relevant-sequences
\[\tilde{z}_{k} = \frac{x(t_k)+\delta_1(t_k)y(t_k)}{\abs{x(t_k)}+\abs{y(t_k)}}\quad\text{and}
\quad \tilde{w}_{k} =
\frac{\delta_2(t_k)x(t_k)+y(t_k)}{\abs{x(t_k)}+\abs{y(t_k)}}.\]
Clearly, $\tilde{z}_k,\tilde{w}_k\notin\Lambda$ for all $k\ge 1$.
Take a subsequence of $\{t_k\}$ if necessary,  we may also assume
that $\frac{x(t_k)}{\abs{x(t_k)}+\abs{y(t_k)}}\rightarrow
\tilde{z}_*$, $\frac{y(t_k)}{\abs{x(t_k)}+\abs{y(t_k)}}\rightarrow
\tilde{w}_*$, $\delta_1(t_k)\to \delta_{1*}\in [0,1-\delta]$ and
$\delta_2(t_k)\to \delta_{2*}\in [0,1]$, as $n\to \infty$. Then one obtains that $\tilde{z}_k\to\tilde{z}_*+\delta_{1*}\tilde{w}_*
\triangleq z_*\notin \Lambda$ and
$\tilde{w}_k\to\delta_{2*}\tilde{z}_*+\tilde{w}_*\triangleq
w_*\notin \Lambda$ as $k\to \infty$, because $\Lambda$ is an open
set. Moreover, the vector $(z_*,w_*)\ne (0,0)$, since $0\le
\delta_{1*}\le 1-\delta$ and $0\le \delta_{2*}\le 1$. (Otherwise, it
follows that $(\tilde{z}_*,\tilde{w}_*)=(0,0),$ which yields that
$1=\frac{\abs{x(t_k)}}{\abs{x(t_k)}+\abs{y(t_k)}}+\frac{\abs{y(t_k)}}{\abs{x(t_k)}+\abs{y(t_k)}}
\to \abs{\tilde{z}_*}+\abs{\tilde{w}_*}=0$, a contradiction.)

\vskip2mm
Without loss of generality, we now assume that $z_*\neq 0$.
For each $k\ge 1$, let \begin{eqnarray*} z_{t_k}(t) =
\frac{x(t+t_k)+\delta_{1*}\cdot
y(t+t_k)}{\abs{x(t_k)}+\abs{y(t_k)}}, \quad t\in \RR.
\end{eqnarray*}
Clearly, $z_{t_k}(t)$ is a nontrivial solution of the equation
\[\dot{x}=A_{t_k}(t)x\triangleq A(t+t_k)x,\]
with the initial value
$z_{t_k}(0)=\tilde{z}_k\to z_*$ as $k\to \infty.$ Recall that $A(t)$ is bounded and uniformly continuous on $\mathbb{R}$. Then one can find a subsequence, still denoted by $\{t_k\}$, such that $A_{t_k}(t)$ converges to $A_*(t)$ uniformly on any compact interval as $k\rightarrow\infty$. Let $z_*(t)$ be the nontrivial solution of
$\dot{x}=A_*(t)x$ with the initial value $z_*(0)=z_*\neq 0$. Then by Lemma \ref{tech} below, one can deduce that $z_*(s)\in\Lambda$ and $\sigma(z_*(s))\equiv{\rm const}$ for all $s\in\mathbb{R}$. In particular, $z_*=z_*(0)\in\Lambda$, a contradiction. Thus we have proved that $W_{m,l}(A)$ is a linear space.

Next we note that
\[\{0\}\subsetneq W_{0,0}(A)\subsetneq W_{0,1}(A)\subsetneq\ldots{}\subsetneq W_{0,n-1}(A),\]
due to the definition of $W_{m,l}(A)$ and Proposition \ref{solutionsbyzeros}.
Since $W_{0,l}(A)$ is a linear space, in order to prove $\dim(W_{m,l}(A))=l-m+1$, it suffices to show that $\dim(W_{0,l}(A))=l+1$ for each $0\leq l\leq n-1$.

Note that $W_{k,k}(A)\subseteq W_{0,l}(A)$ for all $0\leq k\leq l$ and the solutions $x_k(t)\in W_{k,k}(A)$, obtained in Proposition \ref{solutionsbyzeros}, are linearly independent. So, $\dim (W_{0,l}(A))\geq l+1$ (Otherwise, one may find some $0\leq k\leq l$ with $x_k(t)=a_0x_0(t)+\cdots+a_{k-1}x_{k-1}(t)$, which implies that $x_k(t)\in W_{0,k-1}(A)$, contradicting that $x_k(t)\in W_{k,k}(A)$). On the other hand, suppose that $\dim(W_{0,l}(A))> l+1$. Then by induction one can deduce $\dim(W_{0,n-1}(A))>n$, which contradicts the fact $\dim(W_{0,n-1}(A))=n$. Thus we have proved that $\dim(W_{0,l}(A))=l+1$.
\end{proof}

\newtheorem{tech}[dlf]{Lemma}
\begin{tech}\label{tech}
Let $A(t)$ be in \eqref{tri-diagonal-lin} and $A_{\tau}$ be the time $\tau$-shift of $A$, that is $A_{\tau}(t)=A(t+\tau)$. Let also $x(t)$ be a nontrivial solution of \eqref{tri-diagonal-lin}. If there exists a sequence $t_k\rightarrow\infty$ (or $t_k\rightarrow -\infty$) such that $x(t_k)\rightarrow x_*\neq 0$ and $A_{t_k}$ converges to $A_*(t)$ uniformly on any compact interval of $\mathbb{R}$, then the solution $x_*(t)$ of $\dot{x}=A_*(t)x$, with initial value $x_*(0)=x_*$, satisfies
\[x_*(t)\in\Lambda\quad\text{and}\quad\sigma(x_*(t))\equiv {\it const},\]
for any $t\in\mathbb{R}$.
\end{tech}

\begin{proof}
For each $t_k$, we first note that $x(t+t_k)$, $t\in\mathbb{R}$, is a nontrivial solution of $\dot{x}=A_{t_k}(t)x$. It then follows from Grownwall inequality that $x(t+t_k)$ tends to $x_*(t)$ uniformly on any compact interval. So for any $s\in\mathbb{R}$ with $x_*(s)\in\Lambda$, the continuity of $\sigma$ implies that
\begin{equation}\label{limcon}
\sigma(x_*(s))=\lim_{n\rightarrow\infty}\sigma(x(s+t_n))=N_1,
\end{equation}
where the last equality is due to Lemma \ref{DLF}(iii). By the assumption on $A(t)$, it is easy to see that Lemma \ref{DLF} holds for the solution $x_*(t)$ of the equation $\dot{x}=A_*(t)x$. Therefore, \eqref{limcon} yields that $x_*(t)\in\Lambda$ for any $t\in\mathbb{R}$, and hence $\sigma(x_*(t))\equiv{\rm const}$ for all $t\in\mathbb{R}$.
\end{proof}

\newtheorem{rk1}[dlf]{Remark}
\begin{rk1}\label{remark1}
\emph{By Proposition \ref{linear_space}, for each $0\leq m\leq n-1$, the solution $x_m(t)$ obtained in Proposition \ref{solutionsbyzeros} is unique up to a constant multiple. As a consequence, we can normalize $x_m(t)$ so that $|x_m(0)|=1$ and the first coordinate of $x_m(0)$ is positive. Hereafter we always use these normalized solutions in the following context.}
\end{rk1}

We call the spaces $\{W_{m,l}(A)\}_{0\leq m\leq l\leq n-1}$ and these normalized solutions $\{x_m(t)\}_{0\leq m\leq n-1}$, \emph{Floquet spaces} of \eqref{tri-diagonal-lin} and \emph{Floquet solutions} of \eqref{tri-diagonal-lin}, respectively.

\section{Floquet Bundles and Exponential Separation}\label{Floquet_Theory}

Based on the results obtained in previous section, we study in this section the following differential equations with parameter $y\in Y$:
\begin{equation}\label{tri-diagonal_system}
\dot{x}=B(y\cdot t)x, \quad x\in\mathbb{R}^n,\ y\in Y,
\end{equation}
where $y\cdot t$ defines a flow on a compact metric space $(Y,d)$, and $B$ is a continuous matrix-valued function on $Y$. Moreover, $B(y)$ is assumed to be cooperative tridiagonal for each $y\in Y$. Moreover, we denote by $\Phi(t,y)$ as the principal fundamental matrix solution of \eqref{tri-diagonal_system}.

By Proposition \ref{linear_space} and Remark \ref{remark1}, one can obtain the Floquet spaces $W_{l,m}(y)$ and solutions $x_m(t,y)$ associated with each $y\in Y$, where $0\leq l\leq m\leq n-1$. For brevity, we hereafter use the short notations $W_k(y)$ and $x_m(y)$ instead of $W_{k,k}(y)$ and $x_m(0,y)$.

\newtheorem{decomposition}{Proposition}[section]
\begin{decomposition}\label{decomposition}
For $0\leq l\leq m\leq n-1$ and $y\in Y$, we have
\begin{itemize}
\item [\emph{(i)}] $W_{m}(y)=\emph{span}\{x_m(y)\}$ and $W_{l,m}(y)$ has the direct sum decomposition
\begin{equation*}
W_{l,m}(y)=\oplus_{k=l}^{m}W_k(y);
\end{equation*}
\item [\emph{(ii)}] $\Phi(t,y)W_{l,m}(y)=W_{l,m}(y\cdot t)$, for all $t\in\mathbb{R}$;
\item [\emph{(iii)}] $W_{l,m}(y)$ varies continuously with $y\in Y$ as a subspace of $\mathbb{R}^n$.

\end{itemize}

\end{decomposition}

\begin{proof}

(i) This is a direct corollary of the proof of Proposition \ref{linear_space}.

(ii) Due to (i), it suffices to show that, for each $0\leq k\leq n-1$, $\Phi(t,y)W_k(y)=W_k(y\cdot t)$ for all $t\in\mathbb{R}$. Since $W_{k}(y)=\text{span}\{x_k(y)\}$, we only need to prove that $x_k(t,y)\in W_k(y\cdot t)$. To see this, fix $t\in\mathbb{R}$, note that both $x_k(t+s,y)$ and $x_k(s,y\cdot t)$, as functions of $s$, are nontrivial solutions of $\dot{x}=B(y\cdot (t+s))x$. Moreover, noticing $\sigma(x_k(s,y\cdot t))=k=\sigma(x_k(t+s,y))$ for all $t\in\mathbb{R}$, it then follows from Remark \ref{remark1} that there exists a real number $C\neq 0$ such that
\begin{equation}\label{relation}
x_k(t+s,y)=Cx_k(s,y\cdot t)\quad\text{for all }s\in\mathbb{R}.
\end{equation}
By letting $s=0$, we have $x_k(t,y)=Cx_k(0,y\cdot t)\in W_{k}(y\cdot t)$.

(iii) We only need to prove $x_m(y)$ is continuous with $y\in Y$ as an element of $\mathbb{R}^n$, for each $0\leq m\leq n-1$. Fix $0\leq m\leq n-1$ and let $y_k\rightarrow y$ in $Y$. Then $B(y_k\cdot t)$ converges to $B(y\cdot t)$ uniformly for $t$ in any compact interval. Given any subsequence $k'$ satisfying $x_m(y_{k'})\rightarrow x_* ( \text{as } k'\rightarrow\infty)$, by Grownwall inequality again, one has $x_m(t,y_{k'})$ converges to $x(t,y)$ uniformly for $t$ in any compact interval, where $x(t,y)$ is the solution of
\begin{equation*}
\dot{x}=B(y\cdot t)x,
\end{equation*}
satisfying $x(0,y)=x_*$, with $|x_*|=1$ and $(x_*)_1\geq 0$ (because $x_m(y_k')_1>0$ from Remark \ref{remark1}). Moreover, since $\sigma(x_m(t,y_k'))=m$ for all $t\in\mathbb{R}$, we apply Lemma \ref{DLF} to $x(t,y)$ and obtain that $\sigma(x(t,y))=m$ for all $t\in\mathbb{R}$. It then follows from the statement of uniqueness in Remark \ref{remark1} that $x(t,y)=x_m(t,y)$ for all $t\in\mathbb{R}$, which implies that $x_m(y_k')\rightarrow x_*=x(0,y)=x_m(y)$ as $k'\rightarrow\infty$. By arbitrariness of $\{k'\}$, one has $x_m(y_k)\rightarrow x_m(y)$ as $k\rightarrow\infty$. This completes the proof.
\end{proof}

\newtheorem{rk2}[decomposition]{Remark}
\begin{rk2}\label{remark2}
\emph{As a matter of fact, one can further obtain that $C=|x_k(t,y)|$ in \eqref{relation}. Indeed, since $(x_k(0,y\cdot t))_1>0$, $(x_k(0,y))_1>0$ and $x_k(t+s,y), x_k(s,y\cdot t)\in\Lambda$ for all $t,s\in\mathbb{R}$, it entails that $(x_k(t+s,y))_1, (x_k(s,y\cdot t))_1>0$ for all $s\in\mathbb{R}$, and hence $C>0$. Moreover, $C=|x_k(t,y)|$, because $|x_k(0,y\cdot t)|=1$. As a consequence,}
\emph{
\begin{equation}\label{relationco}
    x_k(s,y\cdot t)=\frac{x_k(t+s,y)}{|x_k(t,y)|}, \quad\text {for all }s\in\mathbb{R}.
\end{equation}
}
\end{rk2}

Now let us define the linear skew-product flow $\pi: \mathbb{R}\times\mathbb{R}^n\times Y\rightarrow\mathbb{R}^n\times Y$, associated with system \eqref{tri-diagonal_system} as
\begin{equation}\label{lspf}
\pi(t,x,y)=(\Phi(t,y)x,y\cdot t).
\end{equation}
For each $0\leq m\leq l\leq n-1$, we write $W_{m,l}(Y)=\cup_{y\in Y}W_{m,l}(y)\times\{y\}$ and call them \emph{Floquet bundles}. Clearly, by Proposition \ref{decomposition}, these bundles are $\pi$-invariant and $m\leq\sigma(x)\leq l$, whenever $x\in W_{m,l}(y)\cap\Lambda$, $y\in Y$. Moreover, it is easy to see that $\mathbb{R}^n\times Y=W_{0,k}(Y)\oplus W_{k+1,n-1}(Y)$, where $0\leq k\leq n-2$. In the following, we give the definition of exponential separation between invariant bundles (see \cite{JMSW,Palmer,Po} and references therein).

\newtheorem{exponential_separation}{Definition}[section]
\begin{exponential_separation}\label{des}
\emph{The ordered pair $(X_1,X_2)$ of complementary invariant subbundles of $\mathbb{R}^n\times Y$ is said to be \emph{exponentially separated} for $\pi$ if there exist positive numbers $K$ and $\nu$ such that
\begin{equation}\label{ds_condition}
\frac{|\Phi(t,y)x_2|}{|\Phi(t,y)x_1|}\leq Ke^{-\nu t},\quad t\geq 0,
\end{equation}
for all  $y\in Y$ and $x_1\in X_1(y)$, $x_2\in X_2(y)$ with $|x_1|=|x_2|=1$.}
\end{exponential_separation}

We now present the first main theorem in this section.
\newtheorem{ex}[decomposition]{Theorem}
\begin{ex}\label{dominated_splitting}
For each $0\leq m\leq n-2$, the pair of invariant subbundles $(W_{0,m}(Y),W_{m+1,n-1}(Y))$ is exponentially separated for $\pi$.
\end{ex}

Before we start to prove this Theorem, we need some basic conceptions and definitions.

Given a pair of $(X,\widetilde{X})$ of complementary invariant subbundles of $\mathbb{R}^n\times Y$. For any $y\in Y$, denote by $\Pi(y)$ (resp. $\widetilde{\Pi}(y)$) the projection of $\mathbb{R}^n$ on $X(y)$ along $\widetilde{X}(y)$ (resp. on $\widetilde{X}(y)$ along $X(y)$). In particular, for each $0\leq m\leq n-2$ and the bundle pair $(W_{0,m}(Y),W_{m+1,n-1}(Y))$, we write as $\Pi_m(y)$ (resp. $\widetilde{\Pi}_m(y)$) the projection mapping $\Pi(y)$ (resp. $\widetilde{\Pi}(y)$).

For any $x_1,x_2\in\mathbb{R}^n\backslash\{0\}$, define the equivalence relation $x_1\sim x_2$ if and only if $x_1=\alpha x_2$ for some $\alpha\in\mathbb{R}$ with $\alpha\neq0$. The equivalence class of $x$ will be denoted by $[x]$. Then the linear skew-product flow $\pi$ on $\mathbb{R}^n\times Y$ induces in a natural way a projective flow $\mathbb{P}\pi:\mathbb{R}\times RP^{n-1}\times Y\rightarrow RP^{n-1}\times Y$ as
\begin{equation*}
(t,[x],y)\mapsto ([\Phi(t,y)x],y\cdot t),
\end{equation*}
where $RP^{n-1}$ is the real $(n-1)$-dimensional projective space (see e.g. \cite{Selgrade}).

Let $M\subset RP^{n-1}\times Y$ be a closed invariant subset of $\mathbb{P}\pi$. $M$ is called a \emph{uniformly positive attractor} if it has a neighborhood $U_0$ (called attracting neighborhood) such that, for any neighborhood $V$ of $M$, there is a $T>0$ such that $\mathbb{P}\pi(t,U_0)\subset V$ for all $t>T$.

If $W$ is a vector subbundle of $\mathbb{R}^n\times Y$, then we denote by $\mathbb{P}W$ the projective subbundle associated with $W$. Moreover, the cone of angle $h>0$ about $W$ is the set
\[K(W,h)=\{(x,y)\in\mathbb{R}^n\times Y\ |\ |\widetilde{\Pi}(y)x|\leq h|\Pi(y)x|\}.\]
If we put $\mathbb{P}K(\mathbb{P}W,h)=\{([x],y)\in RP^{n-1}\times Y\ |\ (x,y)\in K(W,h),|x|\neq 0\}$, then $\{\mathbb{P}K(\mathbb{P}W,h): h>0\}$ is a base of the neighborhoods of $\mathbb{P}W$ in $RP^{n-1}\times Y$ (See \cite{Bronshtein}).

\newtheorem{equre}[decomposition]{Lemma}
\begin{equre}\label{transfer}
Let $\pi:\mathbb{R}\times\mathbb{R}^n\times Y\rightarrow\mathbb{R}^n\times Y$ be the skew-product flow defined in \eqref{lspf}. The ordered pair $(X,\widetilde{X})$ of complementary invariant subbundles of $\mathbb{R}^n\times Y$ is exponentially separated if and only if $\mathbb{P}X$ is a uniformly positive attractor for the flow $\mathbb{P}\pi$ on $RP^{n-1}\times Y$.
\end{equre}
\begin{proof}
See Lemma 3 in \cite{Bronshtein}.
\end{proof}

The following lemma shows the uniqueness of exponential separation.

\newtheorem{p1}[decomposition]{Lemma}
\begin{p1}\label{uniqueness}
If the ordered pairs $(X_1,\widetilde{X}_1)$ and $(X_2,\widetilde{X}_2)$ of complementary invariant subbundles of $\mathbb{R}^n\times Y$ are exponentially separated and $\dim(\widetilde{X}_1)=\dim(\widetilde{X}_2)$. Then
\begin{equation*}
    X_1=X_2\quad \text{ and }\quad \widetilde{X}_1=\widetilde{X}_2.
\end{equation*}
\end{p1}

\begin{proof}
See \cite[Lemma A.4]{M}.
\end{proof}

\begin{proof}[Proof of Theorem \ref{dominated_splitting}]
Take any $([x_0],y_0)\in\mathbb{P}W_{0,m}$. It follows from Lemma \ref{DLF} that there is a $\tau>0$ such that $\Phi(\tau,y_0)x_0\in\Lambda$ and $\sigma(\Phi(\tau,y_0)x_0)\leq m$. So, one can find a neighborhood $V$ of $(x_0,y_0)$, with its closure $\overline{V}\subset (\mathbb{R}^n\backslash\{0\})\times Y$, such that $\Phi(\tau,y)x\in\Lambda$ and $\sigma(\Phi(\tau,y)x)\leq m$ for all $(x,y)\in V$. Moreover, one has
\[\lim_{t\rightarrow\infty}\sigma(\Phi(t,y)z)\leq m,\quad\text{ for all }z\in[x] \text{ and } ([x],y)\in \mathbb{P}V.\]
By compactness of $\mathbb{P}W_{0,m}$, we collect a finite neighborhoods $\{\mathbb{P}V_i\}_{i=1}^{k}$ covering $\mathbb{P}W_{0,m}$ in $RP^{n-1}\times Y$, and denote their union by $\mathcal{V}=\bigcup_{i=1}^{k}\mathbb{P}V_i$. Then
\begin{equation}\label{union}
\lim_{t\rightarrow\infty}\sigma(\Phi(t,y)z)\leq m, \quad \text{ for all } z\in[x] \text{ and } ([x],y)\in\mathcal{V}.
\end{equation}
Since $\{\mathbb{P}K(\mathbb{P}W_{0,m},h)\ |\ h>0\}$ is a base of the neighborhoods of $\mathbb{P}W_{0,m}$, one can choose some small $h_0>0$ such that
\[\mathbb{P}K(\mathbb{P}W_{0,m},h_0)\subset\mathcal{V},\]
and hence \eqref{union} is satisfied for all $z\in[x]$ and $([x],y)\in\mathbb{P}K(\mathbb{P}W_{0,m},h_0)$.

We claim that $\mathbb{P}K(\mathbb{P}W_{0,m},h_0)$ is attracting neighborhood of $\mathbb{P}W_{0,m}$. In fact, by \cite[Lemma 1]{Bronshtein}, one only needs to show that given any $([x],y)\in \mathbb{P}K(\mathbb{P}W_{0,m},h_0)$ and any $\varepsilon>0$, there is a $T=T([x],y,\varepsilon)>0$ such that
\[(\Phi(t,y)x,y\cdot t)\in K(W_{0,m},\varepsilon),\quad\text{ for all } t>T.\]
To this end, suppose that there are some $([x_0],y_0)\in\mathbb{P}K(\mathbb{P}W_{0,m},h_0)$, $\varepsilon_0>0$ and $t_n\rightarrow\infty$, satisfying
\begin{equation}\label{contra}
|\widetilde{\Pi}_m(y_0\cdot t_n)\Phi(t_n,y_0)x_0|\geq\varepsilon_0 |\Pi_{m}(y_0\cdot t_n)\Phi(t_n,y_0)x_0|.
\end{equation}
Without loss of generality, we may assume that $y_0\cdot t_n\rightarrow y_*$ and $\frac{\Phi(t_n,y_0)x_0}{|\Phi(t_n,y_0)x_0|}\rightarrow x_*\neq 0$. Similarly as the argument in the proof of Lemma \ref{tech}, we obtain that $\sigma(\Phi(t,y_*)x_*)\equiv m^*$ for all $t\in\mathbb{R}$, which implies that $\sigma(x_*)=m^*$. Noting that $\sigma(x_*)=\lim_{n\rightarrow\infty}\sigma(\Phi(t_n,y_0)x_0)$, \eqref{union} implies that $m^*\leq m$, and hence, $\widetilde{\Pi}_m(y_*)x_*=0$.

On the other hand, by letting $t_n\rightarrow\infty$ in \eqref{contra}, it yields that
\[|\widetilde{\Pi}_m(y_*)x_*|\geq \varepsilon_0|\Pi_m(y_*)x_*|.\]
It then entails that $\Pi_m(y_*)x_*=\widetilde{\Pi}_m(y_*)x_*=0$, which means that $x_*=0$, a contradiction. Thus we have completed the proof of the claim. Consequently, $\mathbb{P}W_{0,m}$ is a uniformly positive attractor of the flow $\mathbb{P}\pi$. Lemma \ref{transfer} then implies that the $(W_{0,m}(Y),W_{m+1,n-1}(Y))$ is exponentially separated.
\end{proof}

Now we go back to the parameterized  linear equation \eqref{tri-diagonal_system}. Fix $y\in Y$ and consider the vectors $x_0(y),\cdots,x_{n-1}(y)$ in Proposition \ref{decomposition}. Then it is not difficult to see that their corresponding solutions $x_0(t,y)\cdots,x_{n-1}(t,y)$ are linearly independent. Thus for any solution $x(t,y)$  of equation \eqref{tri-diagonal_system}, there exist constants $\hat{c}_0,\cdots,\hat{c}_{n-1}$ such that
\begin{equation*}
x(t,y)=\hat{c}_0x_0(t,y)+\cdots+\hat{c}_{n-1}x_{n-1}(t,y),\quad\text{for all }t\in\mathbb{R}.
\end{equation*}
On the other hand, by making use of (i) of Proposition \ref{decomposition} for $y\cdot t$, one can also obtain functions $c_0(t),\cdots,c_{n-1}(t)$ such that
\begin{equation*}
x(t,y)=c_0(t)x_0(y\cdot t)+\cdots+c_{n-1}(t)x_{n-1}(y\cdot t).
\end{equation*}
It then follows from \eqref{relationco} in Remark \ref{remark2} that
\begin{equation}\label{exp3}
c_m(t)=\hat{c}_m|x_m(t,y)|, \quad\text{for } \quad m=0,\cdots,n-1.
\end{equation}
A direct calculation yields that
\begin{equation*}
\frac{d}{dt}|x_m(t,y)|=\frac{x_m^T(t,y)B(y\cdot t)x_m(t,y)}{|x_m(t,y)|},
\end{equation*}
and hence, by \eqref{exp3},
\begin{equation}\label{sepeq}
\dot{c}_m(t)=\lambda_m(y\cdot t)c_m(t),
\end{equation}
where,
\[\lambda_m(y\cdot{}t)=\frac{x_m^T(t,y)B(y\cdot t)x_m(t,y)}
{|x_m(t,y)|^2}.\]
Clearly, $\lambda_m(y)$ is continuous on $Y$ for each $0\leq m\leq n-1$.  As a consequence, we have decoupled \eqref{tri-diagonal_system} into system of 1-D equations \eqref{sepeq}.  Moreover, by Theorem \ref{dominated_splitting}, we have following estimate of the growth rate of such linear equations.

\newtheorem{raterel}[decomposition]{Corollary}
\begin{raterel}\label{raterelation}
Consider the linear skew product flow $\pi$ defined in \eqref{lspf}. Then there exist constants $\beta\geq 0$ and $\gamma>0$ such that
\begin{equation*}\label{integralsep}
    \int_{s}^{t}\lambda_{m+1}(y\cdot \tau)-\lambda_m(y\cdot \tau)d\tau\leq -\gamma (t-s)+\beta
\end{equation*}
for any $s\leq t$ and $m=0,\cdots,n-2$.
\end{raterel}

\begin{proof}
This is a direct corollary from \eqref{ds_condition}, \eqref{exp3} and \eqref{sepeq}.
\end{proof}

We end this section by constructing the relation between the Floquet bundles and the Sacker-Sell spectral bundles of \eqref{tri-diagonal_system}. Recall that $\pi$ is the linear skew-product flow in \eqref{lspf}. Let $\lambda\in\mathbb{R}$ and define $\pi_{\lambda}:\mathbb{R}\times\mathbb{R}^n\times Y\rightarrow\mathbb{R}^n\times Y$ by
\begin{equation}\label{distort}
\pi_{\lambda}(t,x,y)=(\Phi_{\lambda}(t,y)x,y\cdot t),
\end{equation}
where $\Phi_{\lambda}(t,y)=e^{-\lambda t}\Phi(t,y)$.
It is easy to verify that $\pi_{\lambda}$ is also a linear skew-product flow on $\mathbb{R}^n\times Y$. We say $\pi_{\lambda}$ admits an \emph{exponential dichotomy over $Y$} if there is an invariant projector $Q:\mathbb{R}^n\times Y\rightarrow\mathbb{R}^n\times Y$, i.e. $\Phi_{\lambda}(t,y)Q(y)=Q(y\cdot t)\Phi_{\lambda}(t,y)$ and positive constants $K$ and $\alpha$ such that for all $y\in Y$,
\begin{eqnarray*}
|\Phi_{\lambda}(t,y)(1-Q(y))|&\leq& Ke^{-\alpha t},\quad t\geq 0,\\
|\Phi_{\lambda}(t,y)Q(y)|&\leq& Ke^{\alpha t},\quad t\leq 0.
\end{eqnarray*}
$\sum(Y)=\{\lambda\in\mathbb{R}\ |\ \eqref{distort}\text{ has no Exponential Dichotomy over } Y\}$ is called the \emph{Sacker-Sell spectrum} of \eqref{tri-diagonal_system} (or \eqref{lspf}) on $Y$. Further, if $Y$ is connected then its Sacker-Sell spectrum $\sum(Y)$ is of the form (See \cite{SS,SS2}): $\sum(Y)=\bigcup_{i=0}^{l-1}[a_i,b_i]$, where $[a_i,b_i]$ are intervals and they are ordered from right to left, that is, $a_{l-1}\leq b_{l-1}<a_{l-2}\leq b_{l-2}<\cdots<a_0\leq b_0$. We hereafter denote by $V_i$ the associated \emph{spectral bundle} corresponds to the spectrum interval $[a_i,b_i]$, for $i=0,\cdots,l-1$, that is,
\begin{eqnarray*}
    V_i(Y)=\{(x,y)\in\mathbb{R}^n\times Y &|& |\Phi(t,y)x|=o(e^{a_i^-t})\text{ as }t\rightarrow -\infty,\\
    && |\Phi(t,y)x|=o(e^{b_i^+t}) \text{ as }t\rightarrow\infty\},
\end{eqnarray*}
where $a_i^-$, $b_i^+$ are any numbers such that $a_i^-<a_i\leq b_i<b_i^+$. With these notations we present a more delicate decomposition of $V_i(y)$.

\newtheorem{decom2}[decomposition]{Corollary}
\begin{decom2}\label{decom2}
Fix $0\leq i\leq l-1$. Then
\begin{equation*}
V_i(Y)= W_{N+1}(Y)\oplus\cdots\oplus W_{N+M}(Y),
\end{equation*}
where $N=\dim (V_0(Y)\oplus\cdots\oplus V_{i-1}(Y))-1$, $M=\dim V_i(Y)$.
\end{decom2}

\begin{proof}
Fix $0\leq i\leq l-1$. Note that the ordered spectral bundle pair $(V_{0}(Y)\oplus\cdots\oplus V_{i-1}(Y), V_{i}(Y)\oplus\cdots\oplus V_{l-1}(Y))$ is exponentially separated over $Y$, with $\dim (V_0(Y)\oplus\cdots\oplus V_{i-1}(Y))=N+1$. On the other hand, by Theorem \ref{dominated_splitting}, the ordered pair $(W_{0,N}(Y),W_{N+1,n-1}(Y))$ is also exponential separated. By uniqueness (see Lemma \ref{uniqueness}), we obtain that $V_0(Y)\oplus\cdots\oplus V_{i-1}(Y)=W_{0,N}(Y)$ and $V_i(Y)\oplus\cdots\oplus V_{l-1}(Y)=W_{N+1,n-1}(Y)$. Similarly, we can also show that $V_0(Y)\oplus\cdots\oplus V_{i}(Y)=W_{0,N+M}(Y)$ and $V_{i+1}(Y)\oplus\cdots\oplus V_{l-1}(Y)=W_{N+M+1,n-1}(Y)$. Consequently, $V_i(Y)=(V_0(Y)\oplus\cdots\oplus V_i(Y))\cap(V_i(Y)\oplus\cdots\oplus V_{n-1}(Y))=W_{0,N+M}(Y)\cap W_{N+1,n-1}(Y)=W_{N+1}(Y)\oplus\cdots\oplus W_{N+M}(Y)$.
\end{proof}

When there exists $1\le k\le l-1$ such that $b_k<0<a_{k-1}$, then $\pi$ itself admits an exponential dichotomy over $Y$. Let $V^u(y)=V_0(y)\oplus\cdots\oplus V_{k}(y)$, $V^s(y)=V_{k+1}(y)\oplus\cdots\oplus V_{l-1}(y)$. $V^u(y)$ and $V^s(y)$ are called \emph{unstable space} and \emph{stable spaces} of \eqref{tri-diagonal_system} at $y\in Y$, respectively. By virtue of Corollary \ref{decom2}, one has
\begin{equation}\label{dimension}
\begin{split}
\sigma(x)&\leq \dim V^u(y)-1,\quad  \text{for } x\in V^u(y)\cap\Lambda,\\
\sigma(x)&\geq \dim V^u(y), \quad   \text{for } x\in V^s(y)\cap\Lambda,
\end{split}
\end{equation}
 for each $y\in Y$.

\vskip 2mm
Finally, we give the following Lemma due to the continuity of $V^{s,u}(y)$ with respect to $y$, which will be useful in the next section.
\newtheorem{caplinear}[decomposition]{Lemma}
\begin{caplinear}\label{capl}
 Take any $y_1, y_2\in Y$. If the distance $d(y_1,y_2)$ is sufficiently small, then $V^{s,u}(y_1)\oplus V^{u,s}(y_2)=\mathbb{R}^n$.
\end{caplinear}

\begin{proof}
Let $Q(y)$ be the projections on $V^u(y)$ along $V^s(y)$. Since $Q$ is continuous about $y$, $V^u(y)=Q(y)\mathbb{R}^n, V^s(y)=(1-Q(y))\mathbb{R}^n$ vary continuously with respect to $y$. Suppose that there are two sequence $\{y_1^n\},\{y_2^n\}
\subset Y$, satisfying $d(y_1^n,y_2^n)\to 0$ as $n\to \infty$, such that $V^s(y_1^n)\cap V^u(y_2^n)\ne\{0\}$ for any $n$. Without loss of generality, one may assume that $y_i^n\to y_*$ as $n\to \infty$ for any $i=1,2$. Then we can choose unit vectors $w_n\in V^s(y_1^n)\cap V^u(y_2^n)$. By letting $n\to \infty$, the continuity of $V^{u,s}(y)$ with respect to $y$ implies that there is a unit vector $w\in V^s(y_*)\cap V^u(y_*)=\emptyset$, a contradiction. Thus, we have proved that, for $y_1, y_2\in Y$ with $d(y_1,y_2)$ sufficiently small, $V^s(y_1)\cap V^u(y_2)=\{0\}$. Note also that $V^s(y_2)\oplus V^u(y_2)=\mathbb{R}^n$ and $\dim V^u(y_2)=\dim V^u(y_1)$, we obtain that $V^s(y_1)\oplus V^u(y_2)=\mathbb{R}^n$. Similarly, one can prove $V^u(y_1)\oplus V^s(y_2)=\mathbb{R}^n$ provided that $d(y_1,y_2)$ sufficiently small.
\end{proof}

\section{Nonlinear cooperative-competitive tridiagonal systems}\label{application}

In this section, we will apply the Floquet theory obtained in the previous sections to investigate the lifting property of $\omega$-limit sets of the nonautonomous tridiagonal system \eqref{tri-equation-f1}+\eqref{assumption}.

As we mentioned in the introduction, system \eqref{tri-equation-f1}+\eqref{assumption} can be embedded into
a skew-product flow $\Pi_t:\RR^n\times H(f)\to \RR^n\times H(f),$
\begin{equation}\label{skew-fl}
\Pi_t(x_0,g)\mapsto (x(t,x_0,g),g\cdot t),
\end{equation}
 where $x(t,x_0,g)$ is the solution of
\begin{equation}\label{tri-equations-g}
\begin{split}
\dot{x}_1 &=g_1(t,x_1,x_2),\\
\dot{x}_i &=g_i(t,x_{i-1},x_i,x_{i+1}),\quad 2\leq i\leq n-1;\\
\dot{x}_n &=g_n(t,x_{n-1},x_n),\\
\end{split}
\end{equation}
with $x(0;x_0,g)=x_0\in \RR^n$, and $g=(g_1,\cdots,g_n)\in H(f)$. It is also
easy to check that the following two properties:

\vskip 3mm \noindent {\bf (G1)} $\,g$ is $C^1$-admissible;

\vskip2mm \noindent {\bf (G2)} $\quad\dfrac{\partial g_i}{\partial
x_{i+1}}\ge \varepsilon_0, \quad \dfrac{\partial
g_{i+1}}{\partial x_i}\ge \varepsilon_0, \qquad 1\le i\le
n-1, \,(t,x)\in \RR\times \RR^n,$

\noindent hold for each $g\in H(f)$.

\vskip 2mm

\newtheorem{ndlf}{Remark}[section]
\begin{ndlf}\label{re-diffe-x}
{\rm For any $g\in H(f)$, Let $x(t,x_1,g)$ and $x(t,x_2,g)$ be
distinct solutions of {\rm (\ref{tri-equations-g})} on $\mathbb{R}$. If we write $x(t)=x(t,x_1,g)-x(t,x_2,g)$,
then Lemma \ref{DLF} holds for such $x(t)$, in which the element $a_{ij}(t)$ of the matrix-valued function
$A(t)$ can be written as
\[a_{ij}(t)=\int_{0}^{1}\frac{\partial g_i}{\partial x_j}(t,(1-\tau)x(t,x_1,g)+\tau x(t,x_2,g))d\tau.\]
}
\end{ndlf}

In this section, we will assume that $f$ is {\it time-recurrent} (i.e., the time translation flow $(H(f),\mathbb{R}), (g,t)\mapsto g\cdot t$ for
$g\in H(f)$ and $t\in \mathbb{R}$ is minimal). This is satisfied, for instance, when $f$ is a {\it uniformly almost periodic} or, more generally, a {\it uniformly almost automorphic function}, whose definition is given as follows.

\newtheorem{a-p-a-a-function}[ndlf]{Definition}
\begin{a-p-a-a-function}\label{a-p-a-a-function}
{\rm \begin{description}
 \item[{\rm (1)}]
 A function $f\in C(\RR,\RR^n)$ is {\it almost
periodic} if, for any $\varepsilon>0$, the set
$T(\varepsilon):=\{\tau:\abs{f(t+\tau)-f(t)}<\varepsilon,\,\forall
t\in \RR\}$ is relatively dense in $\RR$. $f$ is {\it almost
automorphic} if for any $\{t_n'\}\subset \RR$ there is a
subsequence $\{t_n\}$ and a function $g:\RR\to \RR^n$ such that
$f(t+t_n)\to g(t)$ and $g(t-t_n)\to f(t)$ hold pointwise.

\item[{\rm (2)}] A function $f\in C(\RR\times D,\RR^n)(D\subset \RR^m)$
is {\it uniformly almost periodic} or {\it uniformly almost
automorphic} in $t$ if $f(t,u)$ is bounded and uniformly
continuous on $\RR\times K$ for any compact subset $K\subset D$,
(i.e., $f$ is admissible), and is almost periodic or almost
automorphic in $t\in \RR$.
\end{description}}
\end{a-p-a-a-function}

\noindent A subset $K\subset \mathbb{R}^n\times H(f)$ is {\it invariant}
if $\Pi_t(K)=K$ for every $t\in \RR$. A subset $K\subset \mathbb{R}^n\times H(f)$ is
called {\it minimal} if it is compact, invariant and the only
non-empty compact invariant subset of it is itself. Denote by $p:\mathbb{R}^n\times H(f)\to H(f), (x_0,g)\mapsto g$ the natural flow homomorphism. An invariant compact set $K\subset\mathbb{R}^n\times H(f)$ is called an {\it almost $1$-cover} (resp. {\it $1$-cover}) of $H(f)$, if $p^{-1}(g)\cap K$ is a singleton for at least one $g\in H(f)$ (resp. for any $y\in Y$). \vskip 2mm

The following results, adopted from \cite{Wang}, have already showed the structure of general minimal sets and $\omega$-limit sets of \eqref{skew-fl}.

\newtheorem{minset}[ndlf]{Lemma}
\begin{minset}\label{minset}
\begin{description}
\item{{\rm (i)}} If $E\subset\mathbb{R}^n\times H(f)$ is a minimal set of \eqref{skew-fl}, then $E$ is an almost $1$-cover of $H(f)$;
\item{{\rm (ii)}} Let $E_1,E_2$ be two minimal sets of \eqref{skew-fl}. Then for any $(x_i,g)\in E_i, i=1,2,$ one has $\sigma(x(t,x_1,g)-x(t,x_2,g))= {\it const}$ for all $t\in \mathbb{R}$.
 \item{{\rm (iii)}} Any $\omega$-limit set of \eqref{skew-fl} contains at most two minimal sets.
\end{description}
\end{minset}
\begin{proof}
See \cite[Lemma 4.2, Theorems 3.6 and 4.4]{Wang}.
\end{proof}

Motivated by the work of Shen and Yi \cite{SY,SY2}, we will utilize the obtained Floquet theory in Section \ref{Floquet_Theory} to improve the lifting property of the $\omega$-limit sets which are hyperbolic (see Definition \ref{hyper-def} below).

Let $Y\subset\mathbb{R}^n\times H(f)$ be a compact invariant set of \eqref{skew-fl}. For each $y=(x_0,g)\in Y$, consider the linearized equation of \eqref{tri-equations-g} along the orbit $y\cdot t\triangleq\Pi_t(x_0,g)$:
\begin{equation}\label{hyperbolic}
\dot{z}=B(y\cdot t)z,\quad t\in\mathbb{R},\ z\in\mathbb{R}^n,
\end{equation}
where $B(y\cdot t)=Dg(t,x(t,x_0,g))$ is a matrix-valued function and of cooperative tridiagonal form.

\newtheorem{hyperbolic}[ndlf]{Definition}
\begin{hyperbolic}\label{hyper-def}
{\rm A compact invariant set $Y\subset\mathbb{R}^n\times H(f)$ of \eqref{skew-fl} is called \emph{hyperbolic} if \eqref{hyperbolic} admits an exponential dichotomy over $Y$ and the corresponding projection $Q$ satisfies ${\rm Im} Q(y)\neq\{0\}$ for all $y\in Y$.}
\end{hyperbolic}
Now we are ready to state our main result as follows.

\newtheorem{hy-omega}[ndlf]{Theorem}
\begin{hy-omega}\label{hy-omega}
Consider an $\omega$-limit set $\omega(x_0,g_0)\subset\mathbb{R}^n\times H(f)$ of $(x_0,g_0)\in\mathbb{R}^n\times H(f)$. If $\omega(x_0,g_0)$ is hyperbolic, then it is a $1$-cover of $H(f)$.
\end{hy-omega}

Theorem \ref{hy-omega} immediately follows the lifting-property of hyperbolic omega-limit sets of
two-dimensional competitive or cooperative systems:

\newtheorem{hy-omega-1}[ndlf]{Colrollary}
\begin{hy-omega-1}\label{hy-omega-1}
Consider the two-dimensional nonautonomous competitive ({\rm resp.} cooperative) system
\begin{equation}\label{2D-tri-equation-f1}
\begin{split}
\dot{x}_1 &=f_1(t,x_1,x_2),\\
\dot{x}_2 &=f_2(t,x_{1},x_2),\\
\end{split}
\end{equation}
where $f=(f_1,f_2)$ is a $C^1$-admissible and time-recurrent function satisfying
\begin{equation*}
\dfrac{\partial f_i}{\partial x_{j}}(t,x)\le -\varepsilon_0<0 \quad {\rm \textnormal{ (resp.}}\quad
\dfrac{\partial f_i}{\partial x_{j}}(t,x)\ge \varepsilon_0>0),
 \qquad 1\le i\ne j\le 2,
\end{equation*}
for all $(t,x)\in \RR\times\RR^2$. Let $(x,g)\in \RR^2\times H(f)$ be such that its orbit $\Pi_t(x,g) (t\ge 0)$ is bounded. If the omega-limit set
$\omega(x,g)$ is hyperbolic, then $\omega(x,g)$ is a $1$-cover of $H(f)$.
\end{hy-omega-1}

\newtheorem{main-result}[ndlf]{Remark}
\begin{main-result}
{\rm Theorem \ref{hy-omega} is a natural generalization of the results of
Smillie \cite{Smillie} and Smith \cite{Smith}. Moreover, in a certain
sense, it also extends to higher dimension ($n\ge 3$) the results in spatially
homogeneous cases by Hetzer and Shen \cite{HS}, who
investigated the dynamics of two-dimensional competitive or
cooperative almost periodic systems. In particular, Corollary \ref{hy-omega-1} generalizes the results of
de Mottoni and Schiaffino \cite{MS} and Hale and Somolinos \cite{JHS}, who proved that all
solutions of two-dimensional $T$-periodic competitive or cooperative systems are
asymptotic to $T$-periodic solutions. See also \cite{HS2} and \cite{SW} for extensions of this work.}
\end{main-result}

In order to prove our main results, we first proceed with the characterization of the integer-valued function $\sigma$ on the local invariant manifolds of hyperbolic invariant sets. Our approach is motivated by \cite{SY,SY2}. However, as mentioned in the introduction, we still need more delicate trick to overcome the difficulties that there is no direct characteristic for the critical phase-points which are not in $\Lambda$ (See below the detail in the proof of Theorem \ref{hy-omega}), while for the zero-crossing number, the standard characteristic of $u(\xi)=u_{x}(\xi)=0$, for some $\xi$, can be directly used to analyze such critical situation (see, e.g \cite[Theorem 4.8]{SY}).

\vskip 2mm
Let $Y\subset\mathbb{R}^n\times H(f)$ be a hyperbolic compact invariant set of \eqref{skew-fl}. For any $y=(x_0,g)\in Y$, let $z=x-x(t,x_0,g)$. Then $z$ satisfies the nonlinear equation
\begin{equation}\label{variation}
\dot{z}=B(y\cdot t)z+G(z,y\cdot t),
\end{equation}
where $B(y\cdot t)$ ia as in \eqref{hyperbolic} and $G(z,y\cdot t)=O(|z|^2)$.
Noticing that system \eqref{hyperbolic} admits an exponential dichotomy over $Y$, it follows from standard invariant manifold theory (see \cite{CLL,Henry,SY,YY})
that system \eqref{variation} possesses for each $y\in Y$ a local stable manifold $W^s(y)$ and a local unstable manifold $W^u(y)$; and one can find a constant $C>0$ such that for any $y\in Y$ and $x_s\in W^s(y)$, $x_u\in W^u(y)$,
\begin{eqnarray}\label{longtime}
\begin{split}
|\Psi_t(x_s,y)|&\leq Ce^{-(\alpha/2)t}|x_s|\quad\text{for } t\geq 0,\\
|\Psi_t(x_u,y)|&\leq Ce^{(\alpha/2)t}|x_u|\quad\text{for } t\leq 0,
\end{split}
\end{eqnarray}
where $\Psi_t(\cdot,y)$ is the solution operator of \eqref{variation}. Moreover, they are overflowing invariant in the sense that
\begin{equation}\label{wu-event}
\begin{split}
\Psi_t(W^s(y),y)&\subseteq W^s(y\cdot t)\quad \text{for } t\gg 1,\\
\Psi_t(W^u(y),y) &\subseteq W^u(y\cdot t)\quad \text{for } t\ll -1.
\end{split}
\end{equation}
%
Now for each $y=(x_0,g)\in Y$, we define
\begin{equation*}
\begin{split}
M^s(y)&\triangleq\{x\in\mathbb{R}^n| x-x_0\in W^s(y)\},\\
M^u(y)&\triangleq\{x\in\mathbb{R}^n| x-x_0\in W^u(y)\}.
\end{split}
\end{equation*}
Then $M^s(y)$ and $M^u(y)$ are overflowing invariant to \eqref{tri-equations-g}, that is
\begin{equation*}
\begin{split}
x(t,M^s(y),g)&\subset M^s(y\cdot t)\quad \text{for } t\gg 1,\\
x(t,M^u(y),g)&\subset M^u(y\cdot t)\quad \text{for } t\ll -1.
\end{split}
\end{equation*}
We also note that if $Y$ is moreover connected then $\dim M^{u,s}(y)=\dim V^{u,s}(y)$ are positive integers independent of $y\in Y$, here $V^{u,s}(y)$ are the unstable/stable subspace of \eqref{hyperbolic} defined in the end of last section.

\newtheorem{capnonlinear}[ndlf]{Lemma}
\begin{capnonlinear}\label{capnonlinear}
Let $Y\subset\mathbb{R}^n\times H(f)$ be a connected compact hyperbolic invariant set of \eqref{skew-fl}. Then, for $(x_1,g), (x_2,g)\in Y$ with $|x_1-x_2|$ being sufficiently small, one has $M^s(x_1,g)\cap M^u(x_2,g)\neq\emptyset$.
\end{capnonlinear}

\begin{proof}
By the standard invariant manifold theory(see \cite{CLL,Henry,SY,YY}), for each $y=(x_0,g)\in Y$, there are $C^1$ functions $h^s(\cdot,y): V^s(y)\rightarrow V^u(y)$ and $ h^u(\cdot,y): V^u(y)\rightarrow V^s(y)$ such that $h^{s,u}(z,y)=o(|z|)$, $|(\partial h^{s,u}/\partial z)($ $z,y)|\le C_0<1$ for $z\in V^{s,u}(y)$, and
\[M^s(y)=x_0+W^s(y)=\{x_0+x^s+h^s(x^s,y) | x^s\in V^s(y)\cap B_{\delta_*}(x_0)\},\]
\[M^u(y)=x_0+W^u(y)=\{x_0+x^u+h^u(x^u,y) | x^u\in V^u(y)\cap B_{\delta_*}(x_0)\},\]
where $B_{\delta_*}(x_0)=\{x\in\mathbb{R}^n: |x-x_0|<\delta*\}$ for some $\delta_{_*}>0$.

By Lemma \ref{capl}, one can find a small $\delta_1\in (0,\delta_*)$ such that if
$\abs{x_1-x_2}<\delta_1$, then $V^s(x_1,g)\oplus V^u(x_2,g)=\mathbb{R}^n$. For such $x_1$ and $x_2$, define a function $k:(x_1,x_2,V^s(x_1,g)\oplus V^u(x_2,g))\rightarrow\mathbb{R}^n$, by
\begin{equation}\label{k-represen}
k(x_1,x_2,x_1^s,x_2^u)=x_1+x_1^s+h^s(x_1^s,(x_1,g))-(x_2+x_2^u+h^u(x_2^u,(x_2,g)))
\end{equation} for any
$x_1^s\in V^s(x_1,g)$ and $x_2^u\in V^u(x_2,g)$.

Note that $k(x,x,0,0)=0$ and $\partial k/\partial x_1^s=I+\partial h^s/\partial x_1^s$ is invertible. By implicit function theorem, one can find a $\delta\in (0,\delta_1)$ such that, for each $x_1, x_2\in\mathbb{R}^n$ with $|x_1-x_2|<\delta$ and $x_2^u\in V^u(x_2,g)$ with $|x_2^u|<\delta$, there is a unique $x_1^s\in V^s(x_1,g)$ with $|x_1^s|<\delta$ such that $k(x_1,x_2,x_1^s,x_2^s)=0$ holds. Therefore, by \eqref{k-represen} and the representation of $M^{s,u}(y)$ in the above, one obtains that $M^s(x_1,g)\cap M^u(x_2,g)\neq\emptyset$ whenever $|x_1-x_2|$ sufficiently small. We have completed the proof.
\end{proof}

\newtheorem{class3}[ndlf]{Theorem}
\begin{class3}\label{class3}
Let $Y\subset\mathbb{R}^n\times H(f)$ be a connected compact hyperbolic invariant set of \eqref{skew-fl}. Denote $N=\dim M^u(Y)$. Then for any $y=(x_0,g)\in Y$ and $x^s\in M^s(y)\backslash\{x_0\}$ (resp. $x^u\in M^u(y)\backslash\{x_0\}$), one has  $\sigma(x^s-x_0)\geq N$ (resp. $\sigma(x^u-x_0)\leq N-1$), whenever $x^s-x_0\in \Lambda$ (resp. $x^u-x_0\in \Lambda$).
\end{class3}

\begin{proof}
We prove only for the unstable case and the stable case is similar.

For any $y=(x_0,g)$ and $x^u\in M^u(y)\backslash\{x_0\}$, let $x(t)=x(t,x^u,g)-x(t,x_0,g)$.
Since $x(0)=x^u-x_0\in W^u(y)$, we have $x(t)\in W^u(y\cdot t)$ for $t\ll -1$ by \eqref{wu-event}, and hence
\[x(t)=x_0^u(x(t),y\cdot t)+h^u(x_0^u(x(t),y\cdot t),y\cdot t)\]
for $t\ll -1$, where $x_0^u(x(t),y\cdot t)\in V^u(y\cdot t)$ and $h^u$ is the $C^1$-function defined in the proof of the above Lemma.
Choose any sequence $t_n\rightarrow-\infty$ such that $y\cdot t_n\rightarrow y^*=(x^*,g^*)$. By virtue of the contracting property of $W^u(y\cdot t)$ in the reverse time, it follows that
\[\lim_{n\rightarrow\infty}\frac{x(t_n)}{|x(t_n)|}=\lim_{n\rightarrow\infty}\frac{x_0^u(x(t_n),
y\cdot t_n)}{|x_0^u(x(t_n), y\cdot t_n)|}= x^*\in V^u(y^*).\]
Let $\Phi$ be the principal solution matrix of $\eqref{hyperbolic}_{y^*}$ (i.e., \eqref{hyperbolic} with $y$ replaced by $y^*$). Then $\Phi(t,y^*)x^*$ is a solution of $\eqref{hyperbolic}_{y^*}$. Moreover, for each $t\in \mathbb{R}$, one has $x(t+t_n)/|x(t+t_n)|\rightarrow\Phi(t,y^*)x^*/|\Phi(t,y^*)x^*|$ as $n\rightarrow\infty$. By Lemma \ref{DLF}, there exists a $t_0<0$ such that $\Phi(t_0,y^*)x^*\in\Lambda$. Consequently, $x(t_0+t_n)\in\Lambda$ and $\sigma(x(t_0+t_n))=\sigma(\Phi(t_0,y^*)x^*)$ for all $n$ sufficiently large. Note also that $\Phi(t_0,y^*)x^*\in V^u(y^*\cdot t_0)$. Then by \eqref{dimension}, one obtains that
$\sigma(x(t_0+t_n))\leq N-1$ for all $n$ sufficiently large.

Now recall that $x(t)=x(t,x^u,g)-x(t,x_0,g)$. By virtue of Remark \ref{re-diffe-x} (hence Lemma \ref{DLF} holds for such $x(t)$), we have $\sigma(x^u-x_0)=\sigma(x(0))\le \sigma(x(t_0+t_n))\leq N-1$, whenever $x^u-x_0\in \Lambda$. Thus we have completed the proof.
\end{proof}

\newtheorem{proximal}[ndlf]{Definition}
\begin{proximal}
{\rm  Assume that $Y\subset\mathbb{R}^n\times H(f)$ is a compact invariant set of \eqref{skew-fl}. A pair $(x_1,g),(x_2,g)\in Y$ is said to be \emph{one-sided fiber distal} if $\inf_{t\in\mathbb{R}^{+}}|x(t,x_1,g)-x(t,x_2,g)|>0$ or   $\inf_{t\in\mathbb{R}^-}|x(t,x_1,g)-x(t,x_2,g)|>0$. }
\end{proximal}

\newtheorem{noproximal}[ndlf]{Lemma}
\begin{noproximal}\label{noproximal}
Let $Y\subset\mathbb{R}^n\times H(f)$ be a connect and compact hyperbolic invariant set of \eqref{skew-fl}. Then all pairs in $Y$ are one-sided fiber distal.
\end{noproximal}

\begin{proof}
Suppose that there is a pair $\{(x_1,g_0), (x_2,g_0)\}\in Y$ with $$\inf_{t\in\mathbb{R}^{\pm}}|x(t,x_1,g_0)-x(t,x_2,g_0)|=0.$$ By virtue of Remark \ref{re-diffe-x}, there exists some $t_0\in\mathbb{R}$ such that $x(t_0,x_1,g_0)-x(t_0,x_2,g_0)\in\Lambda$. Then we can choose an $\varepsilon_0>0$ so small that $x(t_0,x_1,g_0)-x(t_0,x_2,g_0)+x\in\Lambda$; and moreover,
\begin{equation}\label{neighbor}
\sigma(x(t_0,x_1,g_0)-x(t_0,x_2,g_0)+x)=\sigma(x(t_0,x_1,g_0)-x(t_0,x_2,g_0))
\end{equation}
whenever $x\in\mathbb{R}^n$ with $|x|<\varepsilon_0$. Choose two sequences $t_k\rightarrow\infty, s_k\rightarrow -\infty$ ( $k\rightarrow\infty$) such that
\[|x(t_k,x_1,g_0)-x(t_k,x_2,g_0)|\rightarrow 0\]
and
\[|x(s_k,x_1,g_0)-x(s_k,x_2,g_0)|\rightarrow 0,\]
as $k\rightarrow\infty$. Consequenctly, Lemma \ref{capnonlinear} implies that for each $k$ sufficiently large, one can find
\[x_+^k\in M^s(x(t_k,x_1,g_0),g_0\cdot t_k)\cap M^u(x(t_k,x_2,g_0),g_0\cdot t_k)\]
and
\[x_-^k\in M^s(x(s_k,x_1,g_0),g_0\cdot s_k)\cap M^u(x(s_k,x_2,g_0) ,g_0\cdot s_k).\]
Using the fact of $x_+^k\in M^u(x(t_k,x_2,g_0),g_0\cdot t_k)$, $x_-^k\in M^s(x(s_k,x_1,g_0),g_0\cdot s_k)$ and \eqref{longtime}, we readily get
\begin{eqnarray*}\label{s1}
|x(s,x_+^k,g_0\cdot t_k)-x(s,x(t_k,x_2,g_0),g_0\cdot t_k)|&\leq & Ce^{(\alpha/2)s}|x_+^k-x(t_k,x_2,g_0)|,\\
|x(t,x_-^k,g_0\cdot s_k)-x(t,x(s_k,x_1,g_0),g_0\cdot s_k)|&\leq & Ce^{-(\alpha/2)t}|x_-^k-x(s_k,x_1,g_0)|,
\end{eqnarray*}
for any $s\leq 0$, $t\geq 0$ and $k$ sufficiently large. In particular, we choose $s=t_0-t_k<0$ and $t=t_0-s_k>0$.
Then one can find some $k_0$ sufficiently large such that
\[|x(t_0-t_{k_0},x_+^{k_0},g_0\cdot t_{k_0})-x(t_0,x_2,g_0)|<\varepsilon_0\]
and
\[|x(t_0-s_{k_0},x_-^{k_0},g_0\cdot s_{k_0})-x(t_0,x_1,g_0)|<\varepsilon_0.\]
Hence, by \eqref{neighbor}, one has
\begin{equation*}
\begin{split}
&\sigma(x(t_0,x_1,g_0)-x(t_0,x_2,g_0))\\
=&\sigma(x(t_0,x_1,g_0)-x(t_0,x_2,g_0)\\
&+x(t_0,x_2,g_0)-x(t_0-t_{k_0},x_+^{k_0},g_0\cdot t_{k_0}))\\
=&\sigma(x(t_0,x_1,g_0)-x(t_0-t_{k_0},x_+^{k_0},g_0\cdot t_{k_0}))\\
\geq &\sigma(x(t_{k_0},x_1,g_0)-x_+^{k_0})\\
\geq & N,
\end{split}
\end{equation*}
where the first inequality comes from Remark \ref{re-diffe-x} and Lemma \ref{DLF}(ii), and the last inequality is due to $x_+^{k_0}\in M^s(x(t_{k_0},x_1,g_0),g_0\cdot t_{k_0})$ and Theorem \ref{class3}. Similarly, one can also obtain that
\begin{equation*}
\begin{split}
& \sigma(x(t_0,x_1,g_0)-x(t_0,x_2,g_0))\\
= &\sigma(x(t_0,x_1,g_0)-x(t_0,x_2,g_0)\\
&+x(t_0-s_{k_0},x_-^{k_0},g_0\cdot s_{k_0})-x(t_0,x_1,g_0))\\
= & \sigma(x(t_0-s_{k_0},x_-^{k_0},g_0\cdot s_{k_0})-x(t_0,x_2,g_0))\\
\leq & \sigma(x(s_{k_0},x_2,g_0)-x_-^{k_0})\\
\leq &  N-1,
\end{split}
\end{equation*}
where the last inequality is due to $x_-^{k_0}\in M^u(x(s_{k_0},x_2,g_0) ,g_0\cdot s_{k_0})$. Thus we have obtained a contradiction, which implies that any pair in $Y$ is one-sided fiber distal.
\end{proof}

\newtheorem{hypmin}[ndlf]{Proposition}
\begin{hypmin}\label{hypmin}
Let $E\subset\mathbb{R}^n\times H(f)$ be a hyperbolic minimal set of \eqref{skew-fl}. Then $E$ is a $1$-cover of $H(f)$.
\end{hypmin}

\begin{proof}
Suppose that $E$ is just an almost $1$-cover (not a $1$-cover) of $H(f)$. Then it follows from minimality of $H(f)$ that there is no one-sided fiber distal pair on $E$, which contradicts to Lemma \ref{noproximal}.
\end{proof}

\vskip 2mm
\begin{proof}[Proof of Theorem {\rm \ref{hy-omega}}.]
According to Proposition \ref{hypmin}, it suffices to show that $\omega(x_0,g_0)$ is minimal. To this end, we suppose that $\omega(x_0,g_0)$ is not minimal. Then Lemma \ref{minset}(iii) implies that $\omega(x_0,g_0)$ can be written as
$\omega(x_0,g_0)=E_1\cup E_2\cup E_{12}$, where $E_i(i=1,2)$ are minimal sets (and hence, they are $1$-covers of $H(f)$ by Proposition \ref{hypmin}). Since $\omega(x,g)$ is connected, $E_{12}\neq\emptyset$. Moreover, for any $(x,g)\in E_{12}$, one has
\begin{equation}\label{ome-alp-asy}
\omega(x,g)\cap
(E_1\cup E_2)\ne \emptyset \,\textnormal{ and }\,\alpha(x,g)\cap
(E_1\cup E_2)\ne \emptyset.
\end{equation}

If $E_1=E_2$, then we pick an $(x_{12},g)\in E_{12}$ and let $(x_1,g)=E_1\cap p^{-1}(g)$. By virtue of
\eqref{ome-alp-asy}, $(x_{12},g)$ and $(x_1,g)$ is not one-sided fiber pair, which contradicts Lemma \ref{noproximal}.

If $E_1\ne E_2$, then again we pick an $(x_{12},g)\in E_{12}$ and let $(x_i,g)\in E_i\cap p^{-1}(g)$ for $i=1,2$.
Due to the same reason in the above paragraph, we may assume without loss of generality that $\omega(x_{12},g)\cap E_1\neq\emptyset$ and $\alpha(x_{12},g)\cap E_2\neq\emptyset$. Since $E_i(i=1,2)$ are $1$-covers of $H(f)$, it is easily seen that $|x(t,x_{12},g)-x(t,x_1,g)|\rightarrow 0$ as $t\rightarrow\infty$, and $|x(t,x_{12},g)-x(t,x_2,g)|\rightarrow 0$ as $t\rightarrow -\infty$. As a consequence, $x(t,x_{12},g)\in M^s(\Pi_t(x_1,g))$ for all $t$ sufficiently positive, and $x(t,x_{12},g)\in M^u(\Pi_t(x_2,g))$ for all $t$ sufficiently negative. It then follows from Theorem \ref{class3} and Remark \ref{re-diffe-x} that, for any $t\in\mathbb{R}$,
\begin{equation}\label{eq1}
\sigma(x(t,x_{12},g)-x(t,x_1,g))\geq N,
\end{equation}
whenever $x(t,x_{12},g)-x(t,x_1,g)\in \Lambda$, and
\begin{equation}\label{eq2}
\sigma(x(t,x_{12},g)-x(t,x_2,g))\leq N-1,
\end{equation}
whenever $x(t,x_{12},g)-x(t,x_2,g)\in \Lambda$. Here $N=\dim M^u(x_{1},g)=M^u(x_{2},g)$.

As mentioned in Remark \ref{re-diffe-x}, $z(t)\equiv x(t,x_2,g)-x(t,x_1,g)$ is a nontrivial solution of the linear equation
$\dot{z}=B(t)z.$ Here $B(t)=\int_{0}^{1}Dg(t,(1-\tau)x(t,x_1,g)+\tau x(t,x_2,g))d\tau$, which is a cooperative tridiagonal matrix. Moreover, $B(t)$ is bounded and uniformly continuous on $\mathbb{R}$, because
$g$ is $C^1$-admissible and $(x_i,g)\in E_i$ with $E_i$ being a $1$-cover for $i=1,2$.

Now we claim that there exist positive numbers $T,\delta>0$ such that
\begin{equation*}
z(t)\in \Lambda_{\delta}\triangleq\{x\in \Lambda:{\rm dist}(x,\Lambda^c)>\delta\}
\end{equation*}
for all $\abs{t}>T$, where $\Lambda^c$ is the complement of $\Lambda$ in $\mathbb{R}^n$. Otherwise, there exists a sequence $t_n\rightarrow\infty$ (or $t_n\rightarrow-\infty$) such that
\begin{equation*}
    \text{dist}(z(t_n),\Lambda^c)\rightarrow 0\quad\text{as }n\rightarrow\infty.
\end{equation*}
Note $x_i\in E_i$, i=1,2. One can choose a subsequence still denoted by $t_n$ such that $x(t_n,x_i,g)\rightarrow w_i\in E_i$. Let $z_*=w_1-w_2\neq 0$. Then $z(t_n)\rightarrow z_*$ as $n\rightarrow\infty$. Moreover, $z_*\notin\Lambda$. By the boundedness and uniform continuity of $B(t)$, we may assume without loss of generality that $B(t+t_n)$ converges to $B_*(t)$ uniformly on any compact interval of $\mathbb{R}$. Then it follows from Lemma \ref{tech} that the solution $z_*(t)$ of $\dot{z}=B_*(t)z$, with initial value $z_*(0)=z_*$ satisfies $z_*(t)\in\Lambda$ for all $t\in\mathbb{R}$, a contradiction to $z_*\notin\Lambda$. Thus we have proved the claim. As a consequence, it entails that
\begin{equation}\label{near-unif}
z(t)+y\in \Lambda,
\end{equation}
 whenever $|y|<\frac{\delta}{2}$ and $\abs{t}>T$.

Recalling that  $|x(t,x_{12},g)-x(t,x_1,g)|\rightarrow 0$ as $t\rightarrow\infty$, it then follows
 from \eqref{eq1} and \eqref{near-unif} that
\begin{equation*}
\begin{split}
\sigma(z(t))
&=\sigma(x(t,x_2,g)-x(t,x_1,g)+x(t,x_1,g)-x(t,x_{12},g))\\
&=\sigma(x(t,x_2,g)-x(t,x_{12},g))\\
&\leq N-1,
\end{split}
\end{equation*}
for all $t$ sufficiently positive.
On the other hand, by \eqref{eq2} and \eqref{near-unif} one can similarly get that
\begin{equation*}
\begin{split}
\sigma(z(t))
&=\sigma(x(t,x_2,g)-x(t,x_1,g)+x(t,x_{12},g)-x(t,x_2,g))\\
&=\sigma(x(t,x_{12},g)-x(t,x_{1},g))\\
&\geq N,
\end{split}
\end{equation*}
for all $t$ sufficiently negative. Hence, we have obtained
a contradiction to Lemma \ref{minset}(ii), which completes the proof of the minimality of $\omega(x_0,g_0)$.
\end{proof}


\begin{thebibliography}{99}
\bibitem{AGT} Ahmad, S., Granadosb, B., and Tineo, A. (2010). On tridiagonal predator-prey systems and a conjecture. 
Nonlinear Analysis: Real World Applications 11, 1878-1881.
\bibitem{Bronshtein} Bronshtein, I. U., and Chernii, V. F. (1978). Linear extensions satisfying Perron's condition. I. Differential Equations 14, 1234-1243.
\bibitem{CLL} Chow, S.-N., Lin, X.-B., and Lu, K. (1991). Soomth invariant foliations in infinite-dimensional spaces. J. Differential Equations 94, no.2, 266-291.
\bibitem{CLM2} Chow, S.-N., Lu, K., and Mallet-Paret, J. (1994). Floquet theory for parabolic differential equations. J. Differential Equations 109, no.1, 147-200.
\bibitem{CLM} Chow, S.-N., Lu, K., and Mallet-Paret, J. (1995). Floquet bundles for scalar parabolic equations. Arch. Rational Mech. Anal. 129, no.3, 245-304.
\bibitem{JHS} Hale, J. K., and Somolinos, A. S. (1983). Competition for fluctuating nutrient. J. Math. Biol. 18, no.3, 255-280.
\bibitem{Henry} Henry, D. (1981). Geometric theory of semilinear parabolic equations. Lecture Notes in Mathematics, 840, Springer-Verlag, Berlin-New York.
\bibitem{HS} Hetzer, G., and Shen, W. (2001). Convergence in Almost Periodic Competition Diffusion Systems. J. Math. Anal. Appl. 262, 307-338.
\bibitem{HS2} Hetzer, G., and Shen, W. (2002). Uniform persistence, coexistence, and extinction in almost periodic/nonautonomous competition diffusion system. SIAM J. Math. Anal. 34, no.1, 204-227.
\bibitem{Hirsch} Hirsch, M. W.(1983). Differential equations and convergence almost everywhere in strongly monotone flows. Contemp. Math., 17, pp.267-285.
\bibitem{MSe}  Mallet-Paret, J., and Sell, G. (1996). Systems of differential delay equations: Floquet multipliers and discrete Lyapunov functions. J. Differential Equations 125, 385-440.
\bibitem{MS} Mallet-Paret, J., and Smith, H. L. (1990). The Poincar\'{e}-Bendixson theorem for monotone cyclic feedback systems. J. Dynam. Differential Equations 2, no.4, 367-421.
\bibitem{Ma} Matano, H. (1982). Nonincrease of the lap-number of a solution for a one-dimensional semilinear parabolic equation. J. Fac. Sci. Univ. Tokyo, Sect. IA Math 29, no.2, 401-441.
\bibitem{M} Mierczy\'{n}ski, J. (1994). The $C^1$ property of carrying simplices for a class of competitive systems of ODEs. J. Differential Equations 111, no.2, 385-409.
\bibitem{JMSW} Mierczy\'{n}ski, J., and Shen, W. (2003). Exponential separation and principal Lyapunov exponent/spectrum for random/nonautonomous parabolic equations. J. Differential Equations 191, no.1, 175-205.
\bibitem{MS} Mottoni, P. de, and Schiaffino, A. (1981). Competition systems with periodic coefficients: a geometric approach. J. Math. Biol. 11, no.3, 319-335.
\bibitem{Ni} Nickel, K. (1962). Gestaltaussagen \"{u}ber L\"{o}sungen parabolischer Differential-gleichungen. J. Reine Angew. Math. 211, 78-94.
\bibitem{Palmer} Palmer, K. J. (1982). Exponential separation, exponential dichotomy and spectral theory for linear systems of ordinary differential equations. J. Differential Equations 46, no.3, 324-345.
\bibitem{Po} Pol\'{a}\v{c}ik, P. (2002). Parabolic equations: Asymptotic behavior and dynamics on invariant manifolds. Handbook of Dynamical Systems, Vol.2, pp. 835-883, North-Holland, Amsterdam.
\bibitem{SS} Sacker, R. J., and Sell, G. R. (1978). A spectral theory for linear differential systems. J. Differential Equations 27, no.3, 320-358.
\bibitem{SS2} Sacker, R. J., and Sell, G. R. (1994) Dichotomies for linear evolutionary equations in Banach spaces. J. Differential Equations 113, no.1, 17-57.
\bibitem{Selgrade} Selgrade, J. F. (1975). Isolated invariant sets for flows on vector bundles. Trans. Amer. Math. Soc. 203, 359-390.
\bibitem{Sell} Sell, G. R. (1971). Topological dynamics and ordinary differential euqations. Van Nostrand Reinhold Co., Van Nostrand Reinhold Mathematical Studies, No.33.
\bibitem{SW} Shen, W., and Wang, Y. (2008). Carrying simplices in nonautonomous and random competitive Kolmogorov systems. J. Differential Equations 245, no.1, 1-29.
\bibitem{SY} Shen, W., and Yi, Y. (1995). Asymptotic almost periodicity of scalar parabolic equations with almost periodic time dependence. J. Differential Equations 122, no.2, 373-397.
\bibitem{SY2} Shen, W., and Yi,Y. (1995). Dynamics of almost periodic scalar parabolic equations. J. Differential Equations 122, no.1, 114-136.
\bibitem{Smillie} Smillie, J. (1984) Competitive and cooperative tridiagonal systems of differential equations. SIAM J. Math. Anal., 15, no.3, 530-534.
\bibitem{Smith} Smith, H. L. (1991). Periodic tridiagonal competitive and cooperative systems of differential equations, SIAM J. Math. Anal. 22, no. 4, 1102-1109.
\bibitem{Wang} Wang, Y. (2007). Dynamics of nonautonomous tridiagonal competitive-cooperative system of differential equaitons. Nonlinearity 20, no.4, 831-843.
\bibitem{YY} Yi, Y. (1993) Stability of integral manifold and orbital attraction of quasi-periodic motions. J. Differential Equations 103, no.2, 278-322.
\end{thebibliography}
\end{document}